\documentclass{amsart}
\usepackage{xypic}
\usepackage[all,cmtip]{xy}

\usepackage{amsmath, amssymb, amsthm, amsfonts, mathrsfs, amsfonts}
\usepackage{cases}
\usepackage{tikz}
\usetikzlibrary{arrows, automata}

\newcommand{\N}{\ensuremath{\mathbb{N}}}

\newtheorem{thm}{Theorem}[section]

\newtheorem{cor}[thm]{Corollary}
\newtheorem{lemma}[thm]{Lemma}
\newtheorem{prop}[thm]{Proposition}
\theoremstyle{remark}

\theoremstyle{definition}
\newtheorem{defin}{Definition}[section]

\newcommand{\on}{\operatorname}

\begin{document}

\title{Tensor products of $n$-complete algebras}
\author{Andrea Pasquali}
\address{Dept. of Mathematics, Uppsala University, P.O. Box 480, 751 06 Uppsala, Sweden}
\email{andrea.pasquali@math.uu.se}
\maketitle

\begin{abstract}
If $A$ and $B$ are $n$- and $m$-representation finite $k$-algebras, then their tensor product $\Lambda = A\otimes_k B$ is not in general 
$(n+m)$-representation finite. However, we prove that if $A$ and $B$ are acyclic and satisfy the weaker assumption of $n$- and $m$-completeness, then $\Lambda$ is $(n+m)$-complete.
This mirrors the fact that taking higher Auslander algebra does not preserve $d$-representation finiteness in general, but it does preserve $d$-completeness.
As a corollary, we get the necessary condition for $\Lambda$ to be $(n+m)$-representation finite which was found by Herschend and Iyama by using a certain twisted fractionally Calabi-Yau property.
\end{abstract}

\section{Introduction}
Higher Auslander-Reiten theory was developed in a series of papers \cite{Iya07}, \cite{Iya07b}, \cite{Iya08} as a tool to study module categories of finite-dimensional algebras.
The idea is to replace all the homological notions in classical Auslander-Reiten theory with higher-dimensional analogs. 
Some early results can be found in \cite{IO11}, \cite{HI11b}. 
This approach has been fruitful in the context of noncommutative algebraic geometry, see for instance
\cite{AIR15}, \cite{HIO14}, \cite{HIMO14}.
Higher Auslander-Reiten theory is also deeply tied with $d$-homological algebra (\cite{GKO13}, \cite{Jas16}, \cite{Jor15}). A presentation of the theory from this point of view can be found in \cite{JK16}.

In this setting, $d$-representation finite algebras were introduced in \cite{Iya11} as a generalisation of hereditary representation finite algebras. 
They are algebras of global dimension at most $d$ that have a $d$-cluster tilting module $M$.
The category $\on{add}M$ has nice homological properties and behaves in many ways like the module category of a hereditary representation finite algebra.
While classification of $d$-representation finite algebras seems far from being achieved, it makes sense to look for examples, and to try to understand how $d$-representation finiteness behaves with respect 
to reasonable operations. 
Notice that in this setting we have more freedom than in the hereditary case, since we are allowed to increase the global dimension and still fall within the scope of the theory.

For instance, in \cite{Iya11} Iyama investigates whether the endomorphism algebra of the $d$-cluster tilting module (called the higher Auslander algebra) is $(d+1)$-representation finite. 
This turns out to be false in general, but a necessary and sufficient condition is given: the only case where it is true is within the tower of iterated higher Auslander algebras of the upper triangular matrix algebra,
so this construction gives only a specific family of examples. 
On the other hand, in the same paper the weaker notion of $d$-complete algebra is introduced and studied. A $d$-complete algebra is an algebra of global dimension at most $d$ that has a module which is 
$d$-cluster tilting in a suitable exact subcategory of the module category. 
It turns out that this weaker notion is preserved under taking higher Auslander algebras, thereby producing many examples of $d$-complete algebras for any $d$. 

Another operation one might investigate is that of taking tensor products over the base field $k$. Indeed, if $k$ is perfect then $\on{gl.dim}A\otimes_k B = \on{gl.dim}A+\on{gl.dim}B$, 
so it makes sense to ask whether the tensor product of an $n$- and an $m$-representation finite algebras is $(n+m)$-representation finite.
This is false in general, and in \cite{HI11} Herschend and Iyama give a necessary and sufficient condition ($l$-homogeneity) for it to be true.

In this paper we prove that the same weaker notion of $d$-completeness which is used in \cite{Iya11} is preserved under tensor products, under the assumption 
of acyclicity. Namely, if $A$ is $n$-complete and acyclic and $B$ is $m$-complete and acyclic, 
then $A\otimes_k B$ is $(n+m)$-complete and acyclic. 
If we assume that $A$ and $B$ are $l$-homogeneous, we recover the result by Herschend and Iyama.
This gives a new way of producing $d$-complete algebras for any $d$. 

The proof we give is structured as follows. We prove that over the tensor product there are $(n+m)$-almost split sequences (using the same construction as in \cite{Pas17}), and moreover
that injective modules have source sequences. Then we use these sequences, combined with the assumption of acyclicity, to prove that the module $T$ in the definition of $(n+m)$-completeness
is tilting. By \cite[Theorem 2.2(b)]{Iya11}, the existence of the above sequences in $T^\perp$ is equivalent to $M$ being $(n+m)$-cluster tilting in $T^\perp$, which is the key point of $(n+m)$-completeness.

In Sections 2 we lay down notation, conventions, and preliminary definitions. Section 3 contains the statement of our main result. 
Section 4 contains the results about $d$-almost split sequences and tensor products which we want to use. Section 5 is dedicated to proving the main theorem, which amounts to checking
that the tensor product satisfies the defining properties of $(n+m)$-complete algebras.
In Section 6 we present some examples.

\section{Notation and conventions}  
Throughout this paper, $k$ denotes a fixed perfect field. All algebras are associative, unital, and finite dimensional over $k$. For an algebra $\Lambda$, $\on{mod}\Lambda$ (respectively
$\Lambda\on{mod}$) denotes the 
category of finitely generated right (left) $\Lambda$-modules. We denote by $D$ the duality 
$D = \on{Hom}_k(-, k)$ between $\on{mod}\Lambda $ and $\Lambda\on{mod}$ (in both directions). 
Subcategories are always assumed to be full and closed under isomorphisms, finite direct sums and summands. 
For $M\in\on{mod}\Lambda$, we denote by $\on{add}M$ the subcategory of $\on{mod}\Lambda$ whose objects are all modules isomorphic to finite direct sums of summands of $M$.
We write $\on{rad}_\Lambda(-,-)$ for the subfunctor of $\on{Hom}_\Lambda(-,-)$ defined by
\begin{align*}
  \on{rad}_{\Lambda}(X, Y) = \left\{ f\in\on{Hom}_{\Lambda}(X, Y)\ |\ \on{id}_X-g\circ f \text{ is invertible } \forall g\in \on{Hom}_\Lambda(Y, X) \right\}.
\end{align*}
Moreover, for $X, Y\in\on{mod}\Lambda$, we write $\on{top}_\Lambda(X ,Y) = \on{Hom}_\Lambda(X, Y) / \on{rad}_{\Lambda}(X, Y)$.
We often write $\on{Hom}$ instead of $\on{Hom}_{\Lambda}$ and similarly for $\on{rad}$ and $\on{top}$ when the context allows it.
We denote by $\mathcal D^b(\Lambda)$ the bounded derived category of $\on{mod}\Lambda$. For a subcategory $\mathcal C$ of $\mathcal D^b(\Lambda)$, we denote by $\on{thick}\mathcal C$ the smallest triangulated subcategory of
$\mathcal D^b(\Lambda)$ containing $\mathcal C$. If $\mathcal C= \on{add}M$ for some $M\in\on{mod}\Lambda\subseteq \mathcal D^b(\Lambda)$, we write $\on{thick}M = \on{thick}(\on{add}M)$.
All tensor products are over $k$, even when the specification is omitted to simplify the notation.

Throughout this section, let $\on{gl.dim}\Lambda\leq d$. 
Then we can define the \emph{higher Auslander-Reiten translations} by 
\begin{align*}
  \tau_d &= D\on{Ext}^d_{\Lambda}(-, \Lambda): \on{mod}\Lambda \to \on{mod}\Lambda\\
  \tau_d^- &= \on{Ext}^d_{\Lambda^{op}}(D-, \Lambda): \on{mod}\Lambda \to \on{mod}\Lambda.
\end{align*}
We are interested in categories associated to tilting modules. 
\begin{defin}
  A $\Lambda$-module $T$ is \emph{tilting} if the following conditions are satisfied:
  \begin{enumerate}
    \item $\on{Ext}^i(T, T) = 0$ for all $i\neq 0$,
    \item there is an exact sequence $0\to \Lambda\to T_0\to \cdots \to T_m\to 0$ for some $m$ with $T_i \in \on{add}T$ for all $i$.
  \end{enumerate}
\end{defin}
The second condition in the definition can be replaced by $$\on{thick}T = \mathcal D^b(\Lambda).$$
For a tilting module $T$, we have an exact subcategory of $\on{mod}\Lambda$
\begin{align*}
  T^\perp = \left\{ X\in \on{mod}\Lambda\ | \ \on{Ext}^i (T, X) = 0 \text{ for every } i\neq 0 \right\}
\end{align*}
We are interested in $d$-cluster tilting subcategories of $T^\perp$.
\begin{defin}
  Let $T$ be a tilting module. A subcategory $\mathcal C$ of $T^\perp$ is called\emph{ $d$-cluster tilting} if
  \begin{align*}
    \mathcal C &= \left\{ X\in T^\perp \ |\ \on{Ext}^i(\mathcal C, X) = 0 \text{ for every }  0<i<d \right\} = \\
    	&= \left\{ X\in T^\perp \ |\ \on{Ext}^i(X, \mathcal C) = 0 \text{ for every }  0<i<d \right\}.
  \end{align*}
\end{defin}
We follow \cite[Definition 1.11]{Iya11} and define the following subcategories of $\on{mod}\Lambda$:
\begin{enumerate}
  \item $\mathcal M  = \mathcal M(\Lambda)= \on{add}\left\{ \tau_d^iD\Lambda\ |\ i\geq 0 \right\}$,
  \item $\mathcal P = \left\{ X\in \mathcal M\ |\  \tau_dX = 0\right\}$, 
  \item $\mathcal M_P = \left\{  X\in \mathcal M\ |\  X \text{ has no nonzero summands in } \mathcal P \right\}$.
  \item $\mathcal M_I = \left\{  X\in \mathcal M\ |\  X \text{ has no nonzero summands in } \on{add}D\Lambda \right\}$.
\end{enumerate}
Let $T_\Lambda$ be a basic module such that $\on{add}T_\Lambda = \mathcal P$.

\begin{defin}
  \label{def:ncomplete}
  An algebra $\Lambda$ is \emph{$d$-complete} if the following conditions hold:
  \begin{enumerate}
    \item [$(A_d)$] $T_\Lambda$ is a tilting module.
    \item [$(B_d)$] $\mathcal M$ is a $d$-cluster tilting subcategory of $T^\perp_\Lambda$,
    \item [$(C_d)$] $\on{Ext}^i (\mathcal M_P, \Lambda) = 0 $ for every $0<i<d$. 
  \end{enumerate}
\end{defin}
Note that condition $(A_d) $ implies that $\tau_d^l = 0$ for large $l$ (\cite[Proposition 1.12(d) and 1.3(c)]{Iya11}).
Note moreover that if $\Lambda$ is $d$-complete then since $\on{gl.dim}\Lambda\leq d$ it follows that $\on{gl.dim}\Lambda\in \left\{ 0, d \right\}$.
This is a generalisation of the notion of $d$-representation finiteness which we use in \cite{Pas17}. 
Without loss of generality, from now on we assume that $\Lambda$ is basic. We write $T$ for $T_\Lambda$ when the context allows it.
Then \cite[Proposition 1.13]{Iya11}
says that ``$d$-representation finite'' is the same as ``$d$-complete with $T = \Lambda$''. 

If $\Lambda$ is $d$-complete, then for every indecomposable injective $I_i$ there is a unique $l_i\in \N$ such that $\tau_d^{l_i-1} I_i \in \mathcal P$, 
and \begin{align*}
  T_\Lambda = \bigoplus_i \tau_d ^{l_i-1} I_i.
\end{align*}

\begin{defin}[\cite{HI11}]
  Let $\Lambda$ be a $k$-algebra of global dimension $d$ such that $\tau_d^l=0$ for $l$ sufficiently large.
  We say that $\Lambda$ is \emph{$l$-homogeneous} if $\tau_d^{l-1}D\Lambda = T_\Lambda$.
\end{defin}

If $\Lambda$ is $d$-complete, this means that $l_i = l$ for every $i$.

Our main result is proved only for acyclic algebras, let us define what we mean by that.
Let $M\in \on{mod}\Lambda$, and let $\mathcal C = \on{add}M$. 
We want to define a preorder on the indecomposable objects $\on{ind}\mathcal C$ of $\mathcal C$. 
For $X, Y\in \on{ind}\mathcal C$, we say $X<Y$ if there is a sequence $(X=X_0, X_1, \dots, X_{m+1} = Y)$ for some $m\geq 0$, such that
$X_i\in \on{ind}\mathcal C$ and $\on{rad}_\Lambda(X_i, X_{i+1})\neq 0$ for all $i$. This defines a transitive relation $<$ on $\on{ind}\mathcal C$.
Notice that we can replace $\on{rad}_\Lambda(X_i, X_{i+1})\neq 0$ with $\on{rad}_{\mathcal C}(X_i, X_{i+1})\neq 0$ above.
\begin{defin}
	The category $\mathcal C$ is \emph{directed} if $<$ is antisymmetric, that is if no indecomposable module $X\in \mathcal C$ satisfies $X<X$.
	If $\mathcal C = \on{add}M$, we say that $M$ is directed. We call the algebra $\Lambda$ \emph{acyclic} if $\Lambda_\Lambda$ is directed.
\end{defin}

\section{Main result}
We now consider the case where $A$ is $n$-complete, $B$ is $m$-complete, and $\Lambda = A\otimes_k B$.
Since $k$ is perfect, we have that $\on{gl.dim} \Lambda = \on{gl.dim} A+ \on{gl.dim} B$.
Moreover, by the K\"unneth formula we have $\tau_{n+m} X\otimes Y= \tau_n X\otimes \tau_m Y$. Since indecomposable
injective $\Lambda$-modules are of the form $X\otimes Y$, it follows that all indecomposable modules in $\mathcal M$ are of this form.
Our main result is the following:
\begin{thm}
  \label{thm:main}
  Let $A, B$ be $n$- respectively $m$-complete acyclic $k$-algebras, with $k$ perfect. Then $A\otimes_k B$ is $(n+m)$-complete and acyclic.
\end{thm}
Note that as far as the author is aware, there are no known examples of $d$-complete algebras which are not acyclic (this is Question 5.9 in \cite{HIO14}).

This result can be applied inductively to construct $d$-complete algebras starting for example from hereditary representation finite algebras and taking tensor products. 
In this sense, it is similar in spirit to \cite[Theorem 1.14 and Corollary 1.16]{Iya11}, where Iyama constructs towers of $d$-complete algebras (with increasing $d$) by taking 
iterated higher Auslander algebras.
The algebra $A\otimes B$ is almost never $(n+m)$-representation finite by the characterisation given by Herschend and Iyama in \cite{HI11}. Our result specialises to their characterisation 
in the acyclic case:

\begin{cor}
  \label{cor:recover}
  Let $A, B$ be $n$- respectively $m$-representation finite acyclic $k$-algebras, with $k$ perfect. Then the following are equivalent:
  \begin{enumerate}
    \item $ A\otimes_k B $ is $(n+m)$-representation finite.
  \item $\exists l\in \N $ such that $ A$ and $B$ are $l$-homogeneous.
\end{enumerate}
Moreover, in this case $A\otimes_k B$ is also $l$-homogeneous.
\end{cor}

It should be noted that there is a choice involved in the definition we gave of $d$-completeness, namely that we take $\mathcal M$ to be the $\tau_d$-completion of $\on{add}D\Lambda$. 
We might as well take $\mathcal M$ to be the $\tau_d^- $-completion of $\on{add}\Lambda$, and call $\Lambda$ $d$-cocomplete if it satisfies the dual conditions to 
$(A_d), (B_d), (C_d)$. Then $\Lambda$ is $d$-complete if and only if $\Lambda^{op}$ is $d$-cocomplete.
Notice that $d$-representation finite is the same as $d$-complete and $d$-cocomplete with the same $\mathcal M$. However, if $A$ and $B$ are $n$- and $m$-representation finite, 
then $A\otimes B$ is $(n+m)$-complete and cocomplete, but in general not with the same $\mathcal M$.

\section{Preparation} 

\subsection{$d$-complete algebras} 
Following \cite{Iya11}, we make some observations about $d$-complete algebras in general.
Fix a finite-dimensional algebra $\Lambda$.

\begin{lemma}
  \label{lem:simplif}
  If $\on{gl.dim} \Lambda\leq d$, the following are equivalent:
  \begin{enumerate}
    \item  $\on{Ext}^i (\mathcal M_P, \Lambda) = 0$ for $0<i<d$
    \item $\on{Ext}^i (\mathcal M_P, \Lambda) = 0$ for $0\leq i<d$.
  \end{enumerate}
\end{lemma}

\begin{proof}
  The only direction to prove follows from \cite[Lemma 2.3(b)]{Iya11}.
\end{proof}

\begin{prop}
  \label{prop:hom}
  If $\Lambda$ is $d$-complete, then 
  \begin{align*}
    \on{Hom}(\tau_d^i D\Lambda, \tau_d^j D\Lambda) = 0
  \end{align*}
  if $i<j$.
\end{prop}

\begin{proof}
  This follows from \cite[Lemma 2.4(e)]{Iya11}.
\end{proof}
We can define \emph{slices} $\mathcal S(i)$ on $\mathcal M$ by saying that $\mathcal S(i) = \on{add}\tau_d^iD\Lambda$. Thus $$\mathcal M = \bigvee_{i\geq 0}\mathcal S(i)$$ (meaning that every object $X\in\mathcal M$ can be written 
uniquely as $X = \bigoplus_{i\geq 0}X_i$ with $X_i \in \mathcal S(i)$) and moreover
$\on{Hom}(\mathcal S(i), \mathcal S(j))  = 0$ if $i<j$.

\begin{lemma}
  \label{lem:equiv}
  If $\Lambda$ is $d$-complete then $\tau_d^\pm$ induce quasi-inverse equivalences $\mathcal M_P\leftrightarrow \mathcal M_I$.
\end{lemma}

\begin{proof}
  This is \cite[Lemma 2.4(b)]{Iya11}.
\end{proof}

\subsection{$d$-almost split sequences}
In the spirit of generalising Auslander-Reiten theory, it is natural to define the higher analog of almost split sequences as follows.
\begin{defin}[Iyama]
  \label{def:dass}
  A complex with objects in a subcategory $\mathcal C$ of $\on{mod}\Lambda$ 
  $$
  \xymatrix{
    C_d\ar[r]^-{f_d}& C_{d-1}\ar[r]^{f_{d-1}}& C_{d-2}\ar[r]^{f_{d-2}}& \cdots
  }
  $$
  is a \emph{source sequence (in $\mathcal C$) of $C_d$} if the following conditions hold:
  \begin{enumerate}
    \item $f_i\in \on{rad}(C_i, C_{i-1})$ for all $i$,
    \item The sequence of functors
      $$
      \xymatrix{
	\cdots\ar[r]^-{-\circ f_{d-2}}&\on{Hom}(C_{d-2}, - )\ar[r]^{-\circ f_{d-1}} & \on{Hom}(C_{d-1}, -)\ar[r]^{-\circ f_d}& \on{rad}(C_d, -)\ar[r]&0
      }
    $$
    is exact on $\mathcal C$.
  \end{enumerate}
    Dually we can define \emph{sink sequences}.
    An exact sequence
	$$
  	\xymatrix{
	  0\ar[r] &C_{d+1}\ar[r] & C_{d-1}\ar[r]&\cdots \ar[r]& C_1\ar[r]& C_0\ar[r]&0
  	}
  	$$
	is an \emph{$d$-almost split sequence} if it is a source sequence of $C_{d+1}$ and a sink sequence of $C_0$.
	We say that such $d$-almost split sequence starts in $C_{d+1}$ and ends in $C_0$.
\end{defin}

 \begin{defin} We say that 
   $\mathcal M = \mathcal M(\Lambda)$ \emph{has $d$-almost split sequences} if for every indecomposable $X\in \mathcal M_I$ (respectively $Y\in \mathcal M_P$) there is an
   $d$-almost split sequence in $\mathcal C$
	\begin{align*}
     0\to X\to C_d\to\cdots\to C_1\to Y\to 0.
   \end{align*}
 \end{defin}

 In this case we must have $X \cong \tau_dY, Y\cong \tau_d^-X$.
 This holds for $d$-complete algebras (\cite[Theorem 2.2(a)(i)]{Iya11}):
 \begin{thm}
   If $\Lambda$ is $d$-complete, then $\mathcal M$ has $d$-almost split sequences.
 \end{thm}

To apply the methods introduced in \cite{Pas17}, we need to rephrase Definition \ref{def:dass} as follows:
for any indecomposable $X\in \mathcal C$
we can define a functor
 $F_X$ on complexes of radical maps by mapping
$$
C_\bullet=
\xymatrix{
	\cdots \ar[r]^{f_{i+1}} & C_i \ar[r]^{f_i} & \cdots \ar[r]^{f_{1}} &  C_0 \ar[r]^{f_0} &\cdots
}
$$
to 
$$
F_X(C_\bullet) =
\xymatrix{
	\cdots \ar[r]^-{f_{i+1}\circ-} &
	\on{Hom}(X, C_i) \ar[r]^-{f_i \circ-} & \cdots \ar[r]^-{f_{1}\circ-} 
	& \on{rad}(X, C_0) \ar[r]^-{f_0\circ-}& \cdots
}
$$
(that is, $F_X$ is the subfunctor of $\on{Hom}(X, -)$ given by replacing $\on{Hom}(X, C_0)$ with $\on{rad}(X,C_0)$). 
Similarly, we can define a subfunctor $G_X$ of the contravariant functor $\on{Hom}(-,X)$ by mapping $C_\bullet$ to
$$
G_X(C_\bullet) = 
\xymatrix{
	\cdots \ar[r]^-{-\circ f_0} &
	\on{Hom}(C_0, X) \ar[r]^-{-\circ f_1} & \cdots \ar[r]^-{-\circ f_{d+1}} 
	& \on{rad}(C_{d+1},X) \ar[r]^-{-\circ f_{d+2}}& \cdots
}
$$

\begin{lemma}
  Let $C_\bullet$ be a complex in $\mathcal C$. Then
  \begin{enumerate}
    \item If $C_i = 0$ for all $i>d+1$, then $C_\bullet$ is a sink sequence if and only if $F_X(C_\bullet)$ is exact for every $X\in \mathcal C$.
    \item If $C_i = 0$ for all $i<0$, then $C_\bullet$ is a source sequence if and only if $G_X(C_\bullet)$ is exact for every $X\in \mathcal C$.
    \item If $C_i = 0$ for all $i>d+1$ and $i<0$, then $C_\bullet$ is $d$-almost split if and only if $F_X(C_\bullet)$ and $G_X(C_\bullet)$ are exact for every $X\in \mathcal C$.
  \end{enumerate}
\end{lemma}

\begin{proof}
  Direct check using the definitions.
\end{proof}

By additivity, in the above Lemma we can replace ``every $X\in\mathcal C$'' by ``every indecomposable $X\in \mathcal C$''.

 Notice that since $d$-almost split sequences come from minimal projective resolutions of a functor $\on{rad}(C_0, -)$, they are uniquely determined by 
 $C_0$ up to isomorphism of complexes.
 Moreover, we have
 \begin{lemma}
   Any map $f_0:C_0\to D_0$ between indecomposables in $\mathcal M_P$ induces a map of complexes $f_\bullet: C_\bullet\to D_\bullet$, where
   $$
   \xymatrix{
     C_\bullet = & 0\ar[r]& C_{d+1}\ar[r]^{g_{d+1}}& \cdots \ar[r]^{g_1}& C_0\ar[r]&0,}
     $$$$
     \xymatrix{
       D_\bullet = & 0\ar[r]& D_{d+1}\ar[r]^{h_{d+1}}& \cdots \ar[r]^{h_1}& D_0\ar[r]&0
  }
   $$
   are the $d$-almost split sequences ending in $C_0$ and $D_0$ respectively, if these exist.
 \end{lemma}

 \begin{proof}
 	The map $f_0g_1: C_1\to D_0$ is a radical morphism, and since
	$$
	\xymatrix{
	  \on{Hom}(C_1, D_1)\ar[r]^{h_1\circ-}& \on{rad}(C_1, D_0)
	}
	$$
   is surjective, there is a map $f_1:C_1\to D_1$ such that $h_1f_1= f_0g_1$. Now assume we have constructed maps $f_j:C_j\to D_j$ that make all diagrams commute, for all $0\leq j<i$ for some $i\geq 2$. 
   We have that 
   $$
   \xymatrix{
   \on{Hom}(C_i, D_i)\ar[r]^{h_i\circ -}& \on{Hom}(C_i, D_{i-1})\ar[r]^{h_{i-1}\circ -}& \on{Hom}(C_i, D_{i-2})
 }
   $$
   is exact in the middle term by assumption. Since $h_{i-1}f_{i-1}g_i = f_{i-2}g_{i-1}g_i = 0$, we have that $f_{i-1}g_i\in \on{ker}(h_{i-1}\circ -) = \on{im}(h_i \circ-)$, that is 
   there is a map $f_i:C_i\to D_i$ such that $f_{i-1}g_i = h_i f_i$. The $f_i$'s we have defined recursively give by construction a map of complexes $f_\bullet: C_\bullet\to D_\bullet$.
 \end{proof}
 The following is a result which appeared in \cite{Pas17} in the setting of $d$-representation finite algebras, and which can be reformulated in the setting of $d$-complete algebras.
\begin{thm}
  Let $\Lambda$ be $d$-complete. Let $X\in \mathcal S(i)$ with $i>0$. Then the $d$-almost split sequence starting in $X$ is isomorphic as a complex to $\on{Cone}\varphi$, where 
  $\varphi: E_\bullet\to F_\bullet$ is a map of complexes, such that:
  \begin{enumerate}
  \item All the maps appearing in $E_\bullet$, $F_\bullet$, and the components of $\varphi$ are radical, 
  \item $E_j\in \mathcal S(i)$ and $F_j\in \mathcal S(i-1)$ for every $j$.
  \end{enumerate}
\end{thm}

\begin{proof}
  This is shown exactly as in \cite[Theorem 2.3]{Pas17}. Namely, one decomposes the modules $M_j$ appearing in the $d$-almost split sequence starting in $X$ as
  $M_j = \bigoplus_{i \geq 0} M_{ij}$ with $M_{ij}\in\mathcal S(i)$.
  One checks using Proposition \ref{prop:hom} that in order for the sequence to be $d$-almost split, all the $M_j$ must be in $\on{add}\left(\tau_d^iD\Lambda\oplus \tau_d^{i-1}D\Lambda\right)$ for some $i$. 
  Now let $E_j = M_{i, j+1}$ and $F_{j} = M_{i-1,j}$. Using that $\on{Hom}(\tau_d^{i-1}D\Lambda, \tau_d^i D\Lambda) = 0 $ one can choose suitable differentials 
  for $E_\bullet$ and $F_{\bullet}$ and a morphism $\varphi_\bullet: E_\bullet\to F_\bullet$ such that $\on{Cone}\varphi$ is the desired sequence.
\end{proof}

We will need a technical lemma:
\begin{lemma}
  \label{lem:nozeros}
  Let 
	$$
  	\xymatrix{
	  0\ar[r] &C_{d+1}\ar[r]^{f_{d+1}} & C_d\ar[r]&\cdots \ar[r]& C_1\ar[r]^{f_1}& C_{0}\ar[r]&0
  	}
  	$$
   be a $d$-almost split sequence. Then for any choice of decomposition of the modules $C_i$ into indecomposables, the corresponding matrices of the maps $f_i$ have no zero column and no zero row.
\end{lemma}

\begin{proof}
  We argue by contradiction. Assume $f_{i}$ has a zero column for $i>1$. Then there is a complex
  $$
  \xymatrix{
    C_{i+1}\ar[r]^-*{\left[
      \begin{smallmatrix}
	f_{i+1}^1\\ f_{i+1}^2
    \end{smallmatrix}\right]
    }
    & C_i^1\oplus C_i^2\ar[r]^-*{ \left[ 
    \begin{smallmatrix}
    f_i^1& 0
\end{smallmatrix}\right]}
& C_{i-1}
  }
  $$
  such that
  $$
  \xymatrix{
    \on{Hom}(C_i^2, C_{i+1})\ar[rr]^-*{\left[ 
      \begin{smallmatrix}
	f_{i+1}^1\circ -\\
	f_{i+1}^2\circ - 
      \end{smallmatrix}
    \right]} 
    && *+{\begin{smallmatrix}
      \on{Hom}(C_i^2, C_i^1)\\
      \oplus\\
      \on{Hom}(C_i^2, C_i^2)
    \end{smallmatrix}}
    \ar[rr]^-*{\left[ \begin{smallmatrix}
      f_i^1\circ- & 0
  \end{smallmatrix}\right]}
  && \on{Hom}(C_i^2, C_{i-1})
  }
  $$
  is exact in the middle, which implies that $f_{i+1}^2\circ -$ is surjective on $\on{Hom}(C_i^2, C_i^2)$, and so
  there is $h\in \on{Hom}(C_i^2, C_{i+1})$ such that $f_{i+1}^2\circ h = \on{id}_{C_{i}^2}$. Since $f_{i+1}^2\in\on{rad}(C_{i+1}, C_i)$, it follows that $C_{i}^2 = 0$ and we are done. 
  For proving the case $i = 1$, just replace $\on{Hom}(C_i^2, C_{i+1})$ with $\on{rad}(C_i^2, C_{i+1})$, and the argument goes through.

  The dual argument, using the fact that $d$-almost split sequences are source, yields the claim for rows.
\end{proof}

\subsection{Tensor products} 

The main tool which allows us to perform homological computations for tensor products is the K\"unneth formula over a field (\cite[VI.3.3.1]{CE56}):
\begin{lemma}
    \label{lem:kunneth}
    If $X_\bullet, Y_\bullet$ are complexes, then there is a functorial isomorphism
    \begin{align*}
      H_i(X_\bullet\otimes Y_\bullet) \cong \bigoplus_{p+q=i} H_p(X_\bullet)\otimes H_q(Y_\bullet).
    \end{align*}
\end{lemma}

Since tensor products of projective resolutions are projective resolutions, we immediately get
\begin{lemma}
  \label{lem:kunnethext}
  If $M_1, M_2\in \on{mod}A$ and $N_1, N_2\in \on{mod}B$, then there is a functorial isomorphism
  \begin{align*} 
    \on{Ext}^i_{A\otimes B}(M_1\otimes N_1, M_2\otimes N_2) \cong \bigoplus_{p+q=i}\on{Ext}^p_A(M_1, M_2)\otimes \on{Ext}^q_B(N_1, N_2).
  \end{align*}
\end{lemma}

The total tensor product of complexes is a functor in a natural way, so we can speak of tensor products of maps of complexes (for a very general treatment of how this is done, 
see \cite[IV.4 and IV.5]{CE56}). 
An important result which is proved in \cite{Pas17} for $d$-representation finite algebras is also true for $d$-complete algebras, namely:
\begin{thm}
  Let $A, B$ be $n$- respectively $m$-complete algebras. Let $\on{Cone}\varphi$ and $\on{Cone}\psi$ be $n$- respectively $m$-almost split sequences starting in
  $\on{add}\tau_n^iDA$ respectively $\on{add}\tau_m^i DB$ for some common $i> 0$.
  Then 
  $\on{Cone}(\varphi\otimes \psi)$ is an $(n+m)$-almost split sequence in $\mathcal M(A\otimes B)$.
\end{thm}
\begin{proof}
  This is proved in the same way as in \cite[Section 3.3]{Pas17}. For convenience, we present the main points of the proof. By definition $\on{Cone}(\varphi\otimes \psi)$ is a complex bounded 
  between 0 and $n+m+1$, it is exact by the K\"unneth formula, and it is easy to check that all maps appearing are radical.
  Now $\varphi: A_\bullet^0 \to A_\bullet^1$ and $\psi: B_\bullet^0 \to B_\bullet^1$, and by assumption we have that $A_j^0 \in\on{add}\tau_n^iDA$, $A_j^1\in \on{add}\tau_n^{i-1}DA$, 
  $B_j^0 \in \on{add}\tau_m^iDB$ and $B_j^1\in \on{add}\tau_m^{i-1}DB$ for every $j$ since $A_j\otimes B_j\in \mathcal M(A\otimes B)$.
  Let now $M\otimes N$ be any indecomposable in $\mathcal M(A\otimes B)$. We need to prove that $F_{M\otimes N}(\on{Cone}(\varphi\otimes \psi))$ is exact. 
  As in \cite[Section 2.3]{Pas17}, for a radical map of radical complexes $\eta:A_\bullet\to B_\bullet$ and a module $X$ we can define $\tilde F_X(\eta) = \eta\circ -: 
  \on{Hom}(X, A_\bullet) \to F_X(B_\bullet)$.
  Then in our setting there is a commutative diagram
  $$
	\xymatrix{
	  \on{Hom}(M, A^0_\bullet)\otimes \on{Hom}(N, B^0_\bullet) \ar[r]^-\cong\ar[d]_*{ \tilde F_M(\varphi)\otimes \tilde F_N(\psi) } 
	  & \on{Hom}(M\otimes N, A^0_\bullet\otimes B^0_\bullet) \ar[d]^*{\tilde F_{M\otimes N}(\varphi\otimes \psi) }\\
	  F_M(A^1_\bullet) \otimes F_N(B^1_\bullet)\ar[r]
	  &  F_{M\otimes N}(A_\bullet^1\otimes B_\bullet^1).
	}
	$$
	Now $F_{M\otimes N}(\on{Cone}(\varphi\otimes \psi))$ is exact if and only if $\tilde F_{M\otimes N}(\varphi\otimes \psi)$ is a quasi-isomorphism. 
	The left map in the diagram $\tilde F_{M}(\varphi)\otimes \tilde F_N(\psi)$ is a quasi-isomorphism since $\on{Cone}(\varphi)$ and $\on{Cone}(\psi)$ 
	are $n$- respectively $m$-almost split sequences. Then it is enough to prove that the bottom map is a quasi-isomorphism, and this is done by showing 
	that its cokernel is isomorphic to 
	$$F_M(A_\bullet^1)\otimes \on{top}(N,B_0^1) \oplus \on{top}(M,A_0^1)\otimes F_N(B_\bullet^1)$$
	and then by easy verification that the above cokernel is exact. The computation of the cokernel is done explicitly in 
	\cite[Section 3.3, pp.660--662]{Pas17}.
\end{proof}

\begin{cor}
  \label{cor:sequences}
  Let $A, B$ be $n$- respectively $m$-complete algebras. Then $\mathcal M( A\otimes B)$ has $(n+m)$-almost split sequences.
\end{cor}

Notice that the above theorem does not require the algebra $A\otimes B$ to be $(n+m)$-representation finite (in which case we know a priori that $(n+m)$-almost split sequences must exist).
In the setting of \cite{Pas17}, this result is about describing the structure of such sequences. In the setting of $d$-complete algebras, this result is used to prove that $(n+m)$-almost split sequences 
exist, whereas it is a priori not clear that they should.

One can also say something about injective modules (which are not the starting point of any $d$-almost split sequence). 

\begin{prop}
  \label{prop:source}
  Let $A, B$ be $n$- respectively $m$-complete algebras, and let $\Lambda = A\otimes B$. Then for every injective $\Lambda$-module $X\otimes Y $ there is a source sequence
  \begin{align*}
    X\otimes Y\to E_{n+m}\to \cdots \to E_1\to 0
  \end{align*}
  in $\mathcal M(\Lambda)$.
\end{prop}

\begin{proof}
   Since $X$ and $Y$ are injective, we have sequences in $\mathcal M(A)$ respectively $\mathcal M(B)$
   \begin{align*} 
   X_\bullet = X\to C_n\to \cdots \to C_1 \to 0\\
   Y_\bullet = Y\to D_m\to \cdots \to D_1\to 0
 \end{align*}
 such that 
 \begin{align*}
   0\to \on{Hom}(C_1, M) \to \cdots \to \on{Hom}(X, M) \to \on{top}(X, M) \to 0,\\
   0\to \on{Hom}(D_1, N) \to \cdots \to \on{Hom}(Y, N) \to \on{top}(Y, N) \to 0
 \end{align*}
	are exact for all indecomposables $M, N$.
Now consider the homology of $X_\bullet\otimes Y_\bullet$. 
\begin{align*}
  H_i (X_\bullet\otimes Y_\bullet) = \bigoplus_{p+q=i}H_p(X_\bullet)\otimes H_q(Y_\bullet) = 
  \begin{cases}
    H_0(X_\bullet)\otimes H_0(Y_\bullet) \text{ if } i = n+m+2\\
    0 \text{ else.}
  \end{cases}
\end{align*}
So we have at least an exact sequence 
\begin{align*}
 X_\bullet\otimes Y_\bullet = X\otimes Y\to \cdots \to C_1\otimes D_1\to 0.
\end{align*}
Apply $\on{Hom}(- , M\otimes N)$ to this sequence and compute homology. 
\begin{align*}
  H_i(\on{Hom}(X_\bullet\otimes Y_\bullet, M\otimes N)) &= H_i\left( \on{Hom}(X_\bullet, M)\otimes \on{Hom}(Y_\bullet, M) \right) =\\
  &= \bigoplus_{p+q=i} H_p(\on{Hom}(X_\bullet, M))\otimes H_q( \on{Hom}(Y_\bullet, M) ) = \\
  &= 
  \begin{cases}
    \on{top}(X, M)\otimes \on{top}(Y, N) \text{ if } i = 0\\
    0 \text{ else.}
  \end{cases}
\end{align*}
We will be done if we prove that $X_\bullet\otimes Y_\bullet$ is source, which amounts now to prove that 
\begin{align*}
  \on{top}(X\otimes Y, M\otimes N) = H_0(\on{Hom}(X_\bullet\otimes Y_\bullet, M\otimes N)) = \on{top}(X, M)\otimes \on{top}(Y, N).
\end{align*}
By tensoring the complexes
\begin{align*}
  0\to\on{rad}(X, M)\to \on{Hom}(X,M)
\end{align*} 
and
\begin{align*}
  0\to\on{rad}(Y, N)\to \on{Hom}(Y, N)  
\end{align*}
and looking at homology, one finds an exact sequence
\begin{align*}
  0&\to \on{rad}(X, M)\otimes \on{Hom}(Y, N) + \on{Hom}(X, M)\otimes \on{rad}(Y, N) \to \\
  &\to\on{Hom}(X, M)\otimes \on{Hom}(Y, N) \to \on{top}(X, M)\otimes \on{top}(Y, N)\to 0.
\end{align*}
Now the middle term is isomorphic to $\on{Hom}(X\otimes Y, M\otimes N)$, and this isomorphism induces an isomorphism between the first term and $\on{rad}(X\otimes Y, M\otimes N)$, 
hence by looking at cokernels we get
\begin{align*}
  \on{top}(X\otimes Y, M\otimes N) &\cong \frac{\on{Hom}(X\otimes Y, M\otimes N)}{\on{rad}(X\otimes Y, M\otimes N)} \\
  &\cong \frac{\on{Hom}(X, M)\otimes \on{Hom}(Y, N)}{\on{rad}(X, M)\otimes \on{Hom}(Y, N)
  + \on{Hom}(X, M)\otimes \on{rad}(Y, N)}\\
  &\cong \on{top}(X, M)\otimes \on{top}(Y, N)
\end{align*}
and we are done.
\end{proof}

\begin{lemma}
  \label{lem:homog}
  Let $A,B$ be $n$- respectively $m$-complete algebras. Then the following are equivalent:
  \begin{enumerate}
    \item $T_{A\otimes B} \cong T_A\otimes T_B$.
    \item $\exists l\in \N$ such that $A$ and $B$ are $l$-homogeneous.
  \end{enumerate}
\end{lemma}

\begin{proof}
  $(2)\Rightarrow (1)$ is clear by definition.

  To prove $(1)\Rightarrow (2)$, assume it does not hold, that is $T_{A\otimes B} \cong T_A\otimes T_B$ but there are $i, j$ such that $l_i\neq l_j$ for 
  the corresponding indecomposable injectives $E_i\in \on{add}DA$ and $F_j\in\on{add}DB$. We can assume 
  that $l_i>l_j$, otherwise the proof is similar.
  Call $X_{ij} =\tau_n^{l_i-1}E_i\otimes \tau_m^{l_j-1}F_j\in \on{add}T_{A\otimes B} $.
  Then $$\tau_{m+n}^{-l_j+1}(X_{ij}) = \tau_{n}^{l_i-l_j}E_i\otimes F_j
  $$
  is not injective, since by assumption $\tau_{n}^{l_i-l_j}E_i$ is not injective. On the other hand, modules in $\mathcal M(A\otimes B)$ which satisfy $\tau_{m+n}X = 0$ are
  precisely the injective $A\otimes B$-modules, and so $\tau_{m+n}^{-l_j+1}(X_{ij})$ is not in $\mathcal M$, contradiction.
\end{proof}

\subsection{Acyclicity}\label{sec:acycl}

We collect here some lemmas about acyclicity which we will use.

\begin{lemma}
  \label{lem:dual}
  The module $\Lambda_\Lambda$ is directed if and only if the module $D_{\Lambda}\Lambda$ is directed.
\end{lemma}

\begin{proof}
  The Nakayama functor induces an equivalence $\nu: \on{add}\Lambda_\Lambda\to \on{add}D_{\Lambda}\Lambda$, and the definition of directedness is invariant under equivalence.
\end{proof}

\begin{lemma}
  Let $\Lambda$ be $d$-complete. Then $\Lambda$ is acyclic if and only if $\mathcal M$ is directed.
\end{lemma}

\begin{proof}
  If $\mathcal M$ is directed, then so is $\on{add}D\Lambda\subseteq \mathcal M$. By Lemma \ref{lem:dual}, $\Lambda$ is then acyclic.

  Conversely, if $\Lambda$ is acyclic then $\on{add}D\Lambda$ is directed by Lemma \ref{lem:dual}, and then so is $\on{add}\tau_d^iD\Lambda$ for any $i$ by Lemma \ref{lem:equiv}.
  Any nonzero map between indecomposables in $\mathcal M$ is either within a slice $\mathcal S(i) = \on{add}\tau_d^iD\Lambda$ or from $\mathcal S(i)$ to $\mathcal S(j)$ with $j<i$. Therefore there can be no cycles within a slice 
  nor cycles that contain modules from different slices and $\mathcal M$ is directed.  
\end{proof}

Acyclicity is well suited to study $d$-almost split sequences.

\begin{lemma}
  \label{lem:wellsuited}
  Let $\Lambda$ be $d$-complete, and let $$
  \xymatrix{
    0\ar[r] & \tau_d X \ar [r] &C_d\ar[r] &\cdots \ar[r]& C_1\ar[r]& X\ar[r]& 0
  }
  $$
  be a $d$-almost split sequence in $\on{mod}\Lambda$.
  Then for every indecomposable summand $Y$ of $\bigoplus_{i = 1}^d C_i$, we have
  \begin{align*}
    \tau_d X< Y<X.
  \end{align*}
\end{lemma}

\begin{proof}
  This follows directly from Lemma \ref{lem:nozeros} and the definition of $<$. 
\end{proof}

Let us now consider acyclicity in relation to tensor products. 

\begin{lemma}
  \label{lem:product}
  The algebras $A$ and $B$ are acyclic if and only if $\Lambda= A\otimes B$ is acyclic.
\end{lemma}

\begin{proof}
  Let us first remark that for $X, X'\in\on{mod}A$ and $Y, Y'\in\on{mod}B$ we have
  \begin{align*}
    \on{rad}(X\otimes Y, X'\otimes Y') = \on{rad}(X, X')\otimes \on{Hom}(Y, Y') + \on{Hom}(X, X')\otimes \on{rad}(Y, Y')
  \end{align*}
  by \cite[Lemma 3.6]{Pas17}.
  Assume $X<X$ in $\on{add}A$ via $X_1, \dots, X_m$. Then for an indecomposable $P\in \on{add}B$ we have that $X\otimes P<X\otimes P$ via $X_1\otimes P, \dots, X_m\otimes P$ since
  \begin{align*} 
    \on{rad}(X_i\otimes P, X_{i+1}\otimes P) \supseteq \on{rad}(X_i, X_{i+1})\otimes \on{End}(P) \neq 0
  \end{align*}
  for all $i$. Therefore if $\Lambda$ is acyclic then $A$ is acyclic.
  By symmetry, if $\Lambda$ is acyclic then $B$ is acyclic as well.

  Let us now prove the converse implication. 
  Assume that $X\otimes Y<X\otimes Y$ in $\on{add}\Lambda$ via $X_1\otimes Y_1, \dots, X_m\otimes Y_m$. 
  We can assume that $\on{rad}(X, X)= 0 = \on{rad}(Y, Y)$. 
  Moreover, it cannot be that $X_i\cong X$ for all $i$ and that $Y_j\cong Y$ for all $j$. Without loss of generality, assume that $X_i\not\cong X$ for some $i$.
  We will prove that $X<X$ via a subsequence $(Z_j)$ of the $X_i$'s.
  We have that $\on{Hom}(X_i, X_{i+1})\neq 0$ for all $i$ by assumption. Set $Z_0 = X$ and $Z_j= X_i$, where $i = \min\left\{ l\ |\ X_l\not\cong Z_{j-1} \right\}$ for $j>0$. 
  By construction, $Z_p = X$ for some $p$ (and for $j>p$, $Z_j$ is not defined). 
  Then we are done, since by construction $\on{Hom}(Z_i, Z_{i+1})\neq 0$ and $Z_i\not\cong Z_{i+1}$ so that $\on{rad}(Z_i,Z_{i+1})\neq 0$ since $Z_{i}, Z_{i+1}$ are indecomposable.
\end{proof}

\section{Proof of main result}
From now on, let $A$ be $n$-complete acyclic, let $B$ be $m$-complete acyclic and let $\Lambda= A\otimes_k B$. We use the notation of Definition \ref{def:ncomplete}.
There are three conditions that need to be checked to prove the main theorem (since we saw in Lemma \ref{lem:product} that $\Lambda$ is acyclic), namely that properties $(A_d), (B_d), (C_d)$ in Definition \ref{def:ncomplete} are preserved under tensor products.
\begin{prop}
  \label{prop:rigid}
  $\on{Ext}^i_\Lambda(\mathcal M, \mathcal M) = 0$ for $0<i<n+m$.
\end{prop}

\begin{proof}
   Let $X\otimes Y\in \mathcal M_P$.
  We have for $i <n+m$
  \begin{align*}
    \on{Ext}^i (X\otimes Y, A\otimes B) &= \bigoplus_{p+q =i} \on{Ext}^p(X, A)\otimes \on{Ext}^q(Y, B) = 0
  \end{align*}
  so we conclude by \cite[Proposition 2.5 (a)]{Iya11}.
 \end{proof}

 By the same formula, $\Lambda$ satisfies condition $(C_{n+m})$:
 \begin{lemma}
   \label{lem:C_d}
   $\on{Ext}^i (\mathcal M_P, \Lambda) = 0$ for all $0<i <n+m$. 
 \end{lemma}

 \begin{proof}
   Use the same formula as in Proposition \ref{prop:rigid}.
 \end{proof}

Notice that since $\tau_{n+m} = \tau_n\otimes \tau_m$ on $\mathcal M$, for sufficienly big $l$ we have $\tau_{n+m}^lD\Lambda = 0$, so $\mathcal M$ has an additive generator.

We now start proving that condition $(A_{n+m})$ holds. 

For $S = S_1\oplus S_2$ with $S_1\in \on{add}T$ and $S_2\in \mathcal M_P$, define $ES = S_1 \oplus \tau_{n+m}S_2$. Note that $E^lD\Lambda = T$ for $l\gg0$. Now fix $S = E^iD\Lambda$ for some $i \geq 0$.
To check condition $(A_{n+m})$ for $\Lambda$, we need some preliminaries.
\begin{lemma}
  \label{lem:rigid}
  If $\on{Ext}^i(S,S) = 0$ for all $i\neq 0$, then $\on{Ext}^i(ES,ES) = 0$ for all $i\neq 0$.
\end{lemma}

\begin{proof}
  Since $\on{Ext}^i_\Lambda(\mathcal M, \mathcal M) = 0$ for $0<i<n+m$, it suffices to check that $\on{Ext}^{n+m}(ES, ES) = 0$. Since $ES = S_1\oplus \tau_{n+m}S_2$, consider first $M_1\otimes N_1\in \on{add}S_1$ and
  $M_2\otimes N_2\in \on{add}ES$. Then 
  \begin{align*}
    \on{Ext}^{n+m}(M_1\otimes N_1,M_2\otimes N_2) = \on{Ext}^n(M_1, M_2)\otimes \on{Ext}^m(N_1, N_2) = 0
  \end{align*}
  since $M_1\otimes N_1\in \on{add}S_1\subseteq \on{add}T$ implies that either $M_1$ or $N_1$ is relative projective in $T_A^\perp$ respectively $T_B^\perp$. 
  This proves that $\on{Ext}^{n+m}(S_1, ES) = 0$. 
  Now let $Y$ be an indecomposable summand of $ES$, and consider $\on{Ext}^{n+m}(\tau_{n+m}S_2, Y)$. If $Y$ is injective, then this is 0.
  Otherwise, $Y = \tau_{n+m} \tau_{n+m}^- Y$ and
  $$\on{Ext}^{n+m}(\tau_{n+m}S_2, Y) = \on{Ext}^{n+m}(S_2,\tau_{n+m}^- Y) = 0$$
  by the assumption.
\end{proof}

\begin{lemma}
  \label{lem:generate}
  If $S$ is tilting then $\on{thick}ES = \mathcal D^b(\Lambda)$.
\end{lemma}

\begin{proof}
  Set $\mathcal S= \on{add}S$.
  For $X\in\on{ind} \mathcal S$, define $h(X)$ to be the height of $X$ with respect to the partial order introduced in Section \ref{sec:acycl} on $\on{ind}\mathcal S$ (here it is crucial that $\Lambda$ be acyclic, 
  which follows from the assumptions on $A$ and $B$ and Lemma \ref{lem:product}), that is 
  $$h(X) = \max\left\{ n\ |\ \exists Y_0<\cdots<Y_n = X, \ Y_i\in \on{ind}\mathcal S\right\}.$$
  Notice that $X>Y$ implies $h(X)>h(Y)$, and the reverse implication holds provided that $X$ and $Y$ are comparable.
  Call $\mathcal C_i =\on{add}\left(\{ES\}\cup \left\{ Y\in\on{ind}\mathcal S \ | \ h(Y)<i\right\} \right)$.
  For $X\in \on{ind}\mathcal S$, if $\tau_{n+m}X = 0$ then $X\in \on{add}ES$. Otherwise, there is an $(n+m)$-almost split sequence
 \begin{align*}
    0\to \tau_{n+m}X\to \cdots \to X\to 0
  \end{align*}
  whose middle terms are in $\on{add}\left(\{ES\}\cup \left\{ Y\in\on{ind}\mathcal S \ | \ Y<X \right\}\right)$ by Lemma \ref{lem:wellsuited}. In particular if $h(X)\leq i$ then the middle terms in the sequence are 
  in $$\on{add}\left(\{ES\}\cup \left\{ Y\in\on{ind}\mathcal S \ | \ h(Y)<i\right\} \right) = \mathcal C_i.$$
  It follows that $\on{thick}\mathcal C_{i+1}\subseteq \on{thick}\mathcal C_i$, so $\on{thick}\mathcal C_j \subseteq \on{thick} \mathcal C_0$ for every $j$. 
  Now $\mathcal C_0 = \on{add}ES$, and
  $\mathcal C_j = \on{add}(ES\oplus S) $ for some $j$, so we get that $\on{thick} ES = \on{thick} \mathcal C_0 = \on{thick}\mathcal C_j = \mathcal D^b(\Lambda)$ as claimed.
\end{proof}

\begin{thm}
  \label{thm:A_d}
  $T=T_{A\otimes B}$ is tilting.
\end{thm}

\begin{proof}
  By Lemma \ref{lem:rigid} and Lemma \ref{lem:generate}, if $S = E^i D\Lambda$ is tilting then $ES = E^{i+1}D\Lambda$ is tilting. 
  Since $D\Lambda$ is tilting, and $T= E^lD\Lambda$ for some $l$, it follows that $T$ is tilting.
\end{proof}

Now we start proving that condition $(B_{n+m})$ holds.
We will use the following result (this is \cite[Theorem 2.2(b)]{Iya11}):
\begin{thm}
  \label{thm:key}
  Let $\Lambda$ be a finite-dimensional $k$-algebra, $d\geq 1$ and $T\in\on{mod}\Lambda$ a tilting module with $\on{proj.dim}T\leq d$.
  Let $\mathcal C = \on{add}C$ be a subcategory of $T^\perp$ such that $\on{Ext}^i_\Lambda(\mathcal C, \mathcal C) = 0$ for $0<i<d$ and $T\oplus D\Lambda\in \mathcal C$.
  Then the following are equivalent:
  \begin{enumerate}
    \item $\mathcal C$ is a $d$-cluster tilting subcategory in $T^\perp$.
    \item Every indecomposable $X\in \mathcal C$ has a source sequence of the form
      \begin{align*}
	X\to C_d\to \cdots \to C_0\to 0
      \end{align*}
      with $C_i\in \mathcal C$ for all $i$.
  \end{enumerate}
\end{thm}

We want to apply this to $\Lambda = A\otimes B$, $\mathcal C= \mathcal M$, $T = T_{A\otimes B}$ and $d = n+m$.

\begin{lemma}
  \label{lem:perp}
  $\mathcal M\subseteq T^\perp$.
\end{lemma}

\begin{proof}
  By Proposition \ref{prop:rigid}, it is enough to check that $\on{Ext}^{n+m}(T , \mathcal M) = 0$. Let $M_1\otimes M_2\in \on{add}T$. 
  Then either $M_1$ or $M_2$ is relative projective in $T_A^\perp$ respectively $T_B^\perp$, so 
  \begin{align*}
    \on{Ext}^{n+m}(M_1\otimes M_2, N_1\otimes N_2) = \on{Ext}^n(M_1, N_1)\otimes \on{Ext}^m(M_2, N_2) = 0
  \end{align*}
  for any $N_1\otimes N_2\in \mathcal M$.
\end{proof}

\begin{thm}
  \label{thm:B_d}
  $\mathcal M$ is an $(n+m)$-cluster tilting subcategory of $T^\perp$.
\end{thm}

\begin{proof}
  By Proposition \ref{prop:rigid} and Lemma \ref{lem:perp}, we can take $\Lambda = A\otimes B$, $\mathcal C = \mathcal M$, $T = T_{A\otimes B}$  and $d= n+m$ in the assumptions of Theorem \ref{thm:key}.
  By Corollary \ref{cor:sequences} and Proposition \ref{prop:source}, condition $(2)$ is satisfied.
  Our claim is then the equivalent statement $(1)$.
\end{proof}

Now we have established everything we need to prove the main result.
\begin{proof}[Proof of Theorem \ref{thm:main}]
  By Theorem \ref{thm:A_d}, Theorem \ref{thm:B_d}, and Lemma \ref{lem:C_d}, we have that $A\otimes B$ satisfies the conditions $(A_{n+m}), (B_{n+m}),(C_{n+m}) $ in the definition of $(n+m)$-complete algebra.
  By Lemma \ref{lem:product}, $A\otimes B$ is acyclic.
\end{proof}

\begin{proof}[Proof of Corollary \ref{cor:recover}]
  By Theorem \ref{thm:main}, $A\otimes B$ is $(n+m)$-complete.
  By \cite[Proposition 1.13]{Iya11}, we have that $T_A\cong A$, $T_B\cong B$ and that $A\otimes B$ is $(n+m)$-representation finite if and only if $T_{A\otimes B} \cong A\otimes B$. 
  By Lemma \ref{lem:homog}, this happens if and only if $A$ and $B$ are $l$-homogeneous for some common $l$.
\end{proof}

\section{Examples}

Let us consider one of the simplest non-trivial examples. Let $A = B = kQ$, where $Q$ is the quiver
\begin{align*} 
  \xymatrix{
    1 & 2.\ar[l]
  }
\end{align*}

Then $\Lambda = A\otimes B$ is the quiver algebra of a commutative square. This algebra is 2-complete, since the factors are 1-representation finite. It is not 2-representation finite since the factors are not homogeneous.
However, $\Lambda$ is representation finite, so we can draw the entire Auslander-Reiten quiver of $\Lambda$. We represent modules by their dimension vector.

$$
  
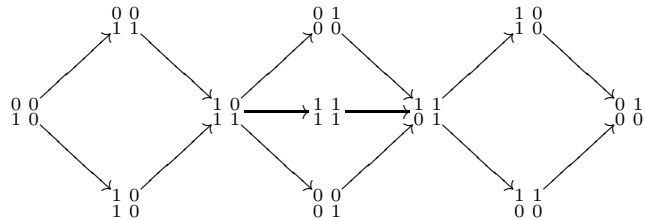
\xymatrix{
  & *{\begin{smallmatrix}
    0&0\\1&1
  \end{smallmatrix}}
  \ar[rd]
  & 
  & *{\begin{smallmatrix}
    0&1\\0&0
  \end{smallmatrix}}
  \ar[dr]
  & 
  & *{\begin{smallmatrix}
    1&0\\1&0
  \end{smallmatrix}}
  \ar[dr]
  \\
    *{\begin{smallmatrix}
      0&0\\1&0
    \end{smallmatrix}}
    \ar[ur]\ar[dr]
    &
    &*{\begin{smallmatrix}
      1&0\\1&1
    \end{smallmatrix}}
    \ar[r]\ar[ur]\ar[dr]
    &
    *{\begin{smallmatrix}
    1&1\\1&1
  \end{smallmatrix}}
  \ar[r]
  &*{\begin{smallmatrix}
  1&1\\0&1
\end{smallmatrix}}
\ar[dr]\ar[ur]
&
& *{\begin{smallmatrix}
0&1\\0&0
\end{smallmatrix}}
\\
  & *{\begin{smallmatrix}
    1&0\\1&0
  \end{smallmatrix}}
  \ar[ru]
  & 
  & *{\begin{smallmatrix}
    0&0\\0&1
  \end{smallmatrix}}
  \ar[ur]
  & 
  & *{\begin{smallmatrix}
    1&1\\0&0
  \end{smallmatrix}}
  \ar[ur]
  }
$$
\vspace{.5cm}

In this case, 
\begin{align*}
  T = 
\begin{smallmatrix}
  0&0\\1&0
\end{smallmatrix} 
\oplus
\begin{smallmatrix}
  1&1\\1&1
\end{smallmatrix}
\oplus \begin{smallmatrix}
  0&1\\0&1
\end{smallmatrix}
\oplus
\begin{smallmatrix}
  1&1\\0&0
\end{smallmatrix}
\end{align*}

and $$\mathcal M = \on{add}M = \on{add}(T\oplus 
\begin{smallmatrix}
  0&1\\0&0
\end{smallmatrix}).$$ 

One can explicitly compute all $\on{Ext}$-groups of all pairs of indecomposables, since we have only finitely many. If we represent by $\otimes$ the indecomposables in $\on{add}T$, 
by $\odot$ the ones in $\mathcal M$ but not in $\on{add}T$, by $\blacksquare$ the ones in $T^\perp$ but not in $\mathcal M$, and by $\cdot$ the ones outside $T^\perp$, we get the following picture:
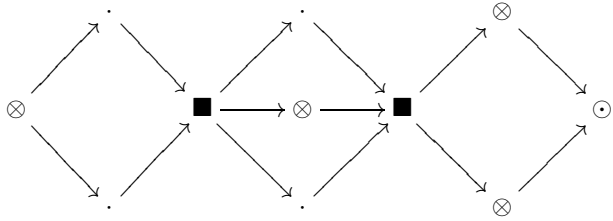
\begin{align*}
  \xymatrix{
    & \cdot\ar[dr]& &\cdot\ar[dr] & &\otimes\ar[dr] &\\
    \otimes \ar[ur]\ar[dr]& &\blacksquare\ar[ur]\ar[dr]\ar[r] & \otimes\ar[r]&\blacksquare \ar[ur]\ar[dr]& &\odot\\
    & \cdot\ar[ur]& &\cdot\ar[ur] & &\otimes\ar[ur] &
  }
\end{align*}
It can be checked that both the indecomposable modules in $ T^\perp\setminus \mathcal M$ have extensions with $M$ on both sides, as it is required by the definition of 
2-cluster tilting. Here we find that $\mathcal M$ is 2-cluster tilting in $T^\perp$.

The Auslander-Reiten quiver of $\on{add}(M)$ is given by 
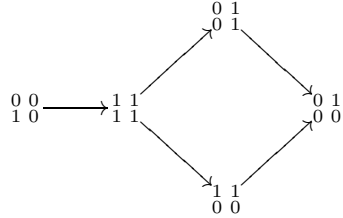
\begin{align*}
  \xymatrix{
  &&*{\begin{smallmatrix}
  0&1\\0&1
\end{smallmatrix}}\ar[dr]&\\
    *{\begin{smallmatrix}
    0&0\\1&0
  \end{smallmatrix}}\ar[r]
  & 
  *{\begin{smallmatrix}
  1&1\\1&1
\end{smallmatrix}}
  \ar[ur]\ar[dr]& 
  & *{\begin{smallmatrix}
  0&1\\0&0
\end{smallmatrix}}\\
&& *{\begin{smallmatrix}
1&1\\0&0
\end{smallmatrix}}\ar[ur]&
 }
\end{align*}
and this is also a picture of the only 2-almost split sequence we have.

As a second example, consider the quiver $Q'$:
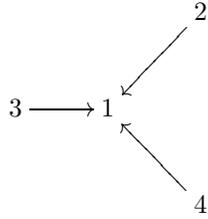
\begin{align*}
  \xymatrix{
    &&2\ar[dl]\\
    3\ar[r]& 1&\\
    &&4\ar[ul]
  }
\end{align*}
and the corresponding path algebra $A' = kQ'$. The Auslander-Reiten quiver of $A'$ looks like
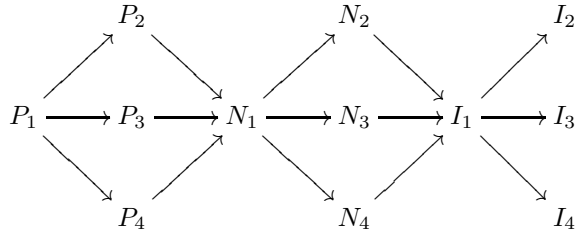
\begin{align*}
  \xymatrix{
    &P_2\ar[dr]&&N_2\ar[dr]&&I_2\\
    P_1\ar[ur]\ar[r]\ar[dr]&P_3\ar[r]&N_1\ar[r]\ar[ur]\ar[dr]& N_3\ar[r]& I_1\ar[r]\ar[ur]\ar[dr]&I_3\\
    &P_4\ar[ur]&&N_4\ar[ur]&&I_4\\
  }
\end{align*}
We take $B'= kQ''$, where $Q''$ is the quiver
\begin{align*}
  \xymatrix{
    a& b\ar[l]&c\ar[l].
  }
\end{align*}
The Auslander-Reiten quiver of $B'$ looks like
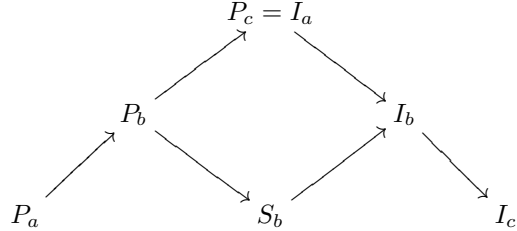
\begin{align*}
  \xymatrix{
    &&P_c = I_a\ar[dr]&&\\
    &P_b\ar[dr]\ar[ur]&&I_b\ar[dr]\\
    P_a\ar[ur]&&S_b\ar[ur]&&I_c
  }
\end{align*}

These algebras are both $1$-representation finite, so in particular they are $1$-complete. Their tensor product $\Lambda' = A'\otimes B'$ is therefore $2$-complete. It is 
not $2$-representation finite since $B'$ is not homogeneous.
In this example, we cannot draw the entire module category of $\Lambda'$, but we still have complete control over
the ``higher Auslander-Reiten quiver'' of $\Lambda'$, that is the Auslander-Reiten quiver of $\on{add}(M)$:
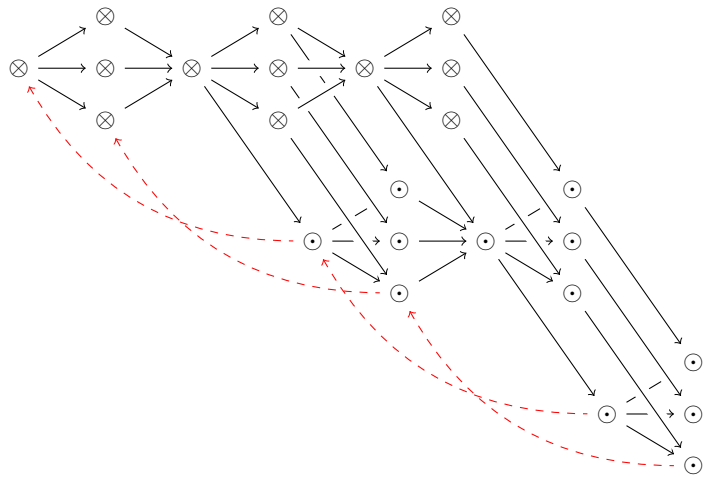
\vspace{.5cm}

\begin{center}
\begin{tikzpicture}[
    x = {(0.5cm, 0)}, y = {(0, 0.3cm)}, z = {(.7cm, -1cm)},
    scale = 2.3, 
    background/.style = { thin }
  ]
  \node (1) at (0, 0, 0){$\otimes$};
  \node (2) at (1, 1, 0){$\otimes$};
  \node (3) at (1, 0, 0){$\otimes$};
  \node (4) at (1, -1, 0){$\otimes$};

  \node (5) at (2, 0, 0){$\otimes$};
  \node (6) at (3, 1, 0){$\otimes$};
  \node (7) at (3, 0, 0){$\otimes$};
  \node (8) at (3, -1, 0){$\otimes$};

  \node (9) at (4, 0, 0){$\otimes$};
  \node (10) at (5, 1, 0){$\otimes$};
  \node (11) at (5, 0, 0){$\otimes$};
  \node (12) at (5, -1, 0){$\otimes$};

  \node (13) at (2, 0, 1){$\odot$};
  \node (14) at (3, 1, 1){$\odot$};
  \node (15) at (3, 0, 1){$\odot$};
  \node (16) at (3, -1, 1){$\odot$};

  \node (17) at (4, 0, 1){$\odot$};
  \node (18) at (5, 1, 1){$\odot$};
  \node (19) at (5, 0, 1){$\odot$};
  \node (20) at (5, -1, 1){$\odot$};

  \node (21) at (4, 0, 2){$\odot$};
  \node (22) at (5, 1, 2){$\odot$};
  \node (23) at (5, 0, 2){$\odot$};
  \node (24) at (5, -1, 2){$\odot$};

  \path[<-](1)edge[dashed,red, bend right] (13);
  \path[<-](4)edge[dashed,red, bend right] (16);
  \path[<-](13)edge[dashed,red, bend right] (21);
  \path[<-](16)edge[dashed,red, bend right] (24);


  \path[->] (1) edge (2) edge (3) edge (4);
  \path[<-] (5) edge(2) edge (3) edge (4);

  \path[->] (5)edge(6)edge(7) edge (8);
  \path[<-] (9)edge(6)edge(7) edge (8);

  \path[->] (9)edge(10)edge(11) edge (12);

  \path[->] (13) edge (14) edge (15) edge (16);
  \path[<-] (17) edge (14) edge (15) edge (16);

  \path[->] (17)edge(18)edge(19) edge (20);

  \path[->] (21)edge(22)edge(23) edge (24);

  \path[->] (5) edge (13);
  \path[->] (6) edge (14);
  \path[->] (7) edge (15);
  \path[->] (8) edge (16);
 
  \path[->] (9) edge (17);
  \path[->] (10) edge (18);
  \path[->] (11) edge (19);
  \path[->] (12) edge (20);

  \path[->] (17) edge (21);
  \path[->] (18) edge (22);
  \path[->] (19) edge (23);
  \path[->] (20) edge (24); 

 \path[->] (21)edge[-,line width=6pt,draw=white](22)edge[-,line width=6pt,draw=white](23) edge[-,line width=6pt,draw=white] (24);
  \path[->] (21)edge(22)edge(23) edge (24);

  \path[->] (17) edge[-,line width=6pt,draw=white]  (21);
  \path[->] (18) edge[-,line width=6pt,draw=white]  (22);
  \path[->] (19) edge[-,line width=6pt,draw=white]  (23);
  \path[->] (20) edge[-,line width=6pt,draw=white]  (24);
  \path[->] (17) edge (21);
  \path[->] (18) edge (22);
  \path[->] (19) edge (23);
  \path[->] (20) edge (24);

  \path[->] (13) edge[-,line width=6pt,draw=white] (14) edge[-,line width=6pt,draw=white] (15) edge[-,line width=6pt,draw=white] (16);
  \path[<-] (17) edge[-,line width=6pt,draw=white] (14) edge[-,line width=6pt,draw=white] (15) edge[-,line width=6pt,draw=white] (16);
  \path[->] (17)edge[-,line width=6pt,draw=white](18)edge[-,line width=6pt,draw=white](19) edge[-,line width=6pt,draw=white] (20);
  \path[->] (13) edge (14) edge (15) edge (16);
  \path[<-] (17) edge (14) edge (15) edge (16);
  \path[->] (17)edge(18)edge(19) edge (20); 

  \path[->] (5) edge[-,line width=6pt,draw=white] (13);
  \path[->] (6) edge[-,line width=6pt,draw=white]  (14);
  \path[->] (7) edge[-,line width=6pt,draw=white]  (15);
  \path[->] (8) edge[-,line width=6pt,draw=white]  (16);
  \path[->] (5) edge (13);
  \path[->] (6) edge (14);
  \path[->] (7) edge (15);
  \path[->] (8) edge (16);
  \path[->] (9) edge[-,line width=6pt,draw=white]  (17);
  \path[->] (10) edge[-,line width=6pt,draw=white]  (18);
  \path[->] (11) edge[-,line width=6pt,draw=white]  (19);
  \path[->] (12) edge[-,line width=6pt,draw=white]  (20);
  \path[->] (9) edge (17);
  \path[->] (10) edge (18);
  \path[->] (11) edge (19);
  \path[->] (12) edge (20);
 
  \path[->] (1) edge[-,line width=6pt,draw=white] (2) edge[-,line width=6pt,draw=white] (3) edge[-,line width=6pt,draw=white] (4);
  \path[<-] (5) edge[-,line width=6pt,draw=white](2) edge[-,line width=6pt,draw=white] (3) edge[-,line width=6pt,draw=white] (4);
  \path[->] (5)edge[-,line width=6pt,draw=white](6)edge[-,line width=6pt,draw=white](7) edge[-,line width=6pt,draw=white] (8);
  \path[<-] (9)edge[-,line width=6pt,draw=white](6)edge[-,line width=6pt,draw=white](7) edge[-,line width=6pt,draw=white] (8);
  \path[->] (9)edge[-,line width=6pt,draw=white](10)edge[-,line width=6pt,draw=white](11) edge[-,line width=6pt,draw=white] (12);
  \path[->] (1) edge (2) edge (3) edge (4);
  \path[<-] (5) edge(2) edge (3) edge (4);
  \path[->] (5)edge(6)edge(7) edge (8);
  \path[<-] (9)edge(6)edge(7) edge (8);
  \path[->] (9)edge(10)edge(11) edge (12); 

\end{tikzpicture}
\end{center}
Here the dashed arrows represent $\tau_2$, and we have drawn them only between some modules to avoid clogging the picture. We have again written $\otimes$ for indecomposable summands of $T$, and $\odot$ for the other
indecomposable summands of 
$M$. It should be clear from the picture which module corresponds to which node.

Notice that this example presents some regularity which is not to be expected in general, since we have taken $A'$ to be homogeneous.
Moreover, in this example (and in general) we cannot directly check that arbitrary modules in $\on{mod}\Lambda'$ which are in $T^\perp$ have extensions on both sides with $\mathcal M$.
\vspace{1cm}

\paragraph{\textbf{Acknowledgement.}} 
I am especially thankful to his advisor Martin Herschend for the many helpful comments and suggestions. I would also like to thank an anonymus referee for his or her comments.

\bibliographystyle{alpha}
\bibliography{Bibliography}{}

\begin{thebibliography}{HIMO14}

\bibitem[AIR15]{AIR15}
Claire Amiot, Osamu Iyama, and Idun Reiten.
\newblock Stable categories of {C}ohen-{M}acaulay modules and cluster
  categories.
\newblock {\em Amer. J. Math.}, 137(3):813--857, 2015.

\bibitem[CE56]{CE56}
Henri Cartan and Samuel Eilenberg.
\newblock {\em Homological algebra}.
\newblock Princeton University Press, Princeton, N. J., 1956.

\bibitem[GKO13]{GKO13}
Christof Geiss, Bernhard Keller, and Steffen Oppermann.
\newblock {$n$}-angulated categories.
\newblock {\em J. Reine Angew. Math.}, 675:101--120, 2013.

\bibitem[HI11a]{HI11}
Martin Herschend and Osamu Iyama.
\newblock {$n$}-representation-finite algebras and twisted fractionally
  {C}alabi-{Y}au algebras.
\newblock {\em Bull. Lond. Math. Soc.}, 43(3):449--466, 2011.

\bibitem[HI11b]{HI11b}
Martin Herschend and Osamu Iyama.
\newblock Selfinjective quivers with potential and 2-representation-finite
  algebras.
\newblock {\em Compos. Math.}, 147(6):1885--1920, 2011.

\bibitem[HIMO14]{HIMO14}
Martin Herschend, Osamu Iyama, Hiroyuki Minamoto, and Steffen Oppermann.
\newblock Representation theory of {G}eigle-{L}enzing complete intersections.
\newblock {\em arXiv:1409.0668v1}, 2014.

\bibitem[HIO14]{HIO14}
Martin Herschend, Osamu Iyama, and Steffen Oppermann.
\newblock {$n$}-representation infinite algebras.
\newblock {\em Adv. Math.}, 252:292--342, 2014.

\bibitem[IO11]{IO11}
Osamu Iyama and Steffen Oppermann.
\newblock {$n$}-representation-finite algebras and {$n$}-{APR} tilting.
\newblock {\em Trans. Amer. Math. Soc.}, 363(12):6575--6614, 2011.

\bibitem[Iya07a]{Iya07b}
Osamu Iyama.
\newblock Auslander correspondence.
\newblock {\em Adv. Math.}, 210(1):51--82, 2007.

\bibitem[Iya07b]{Iya07}
Osamu Iyama.
\newblock Higher-dimensional {A}uslander-{R}eiten theory on maximal orthogonal
  subcategories.
\newblock {\em Adv. Math.}, 210(1):22--50, 2007.

\bibitem[Iya08]{Iya08}
Osamu Iyama.
\newblock Auslander-{R}eiten theory revisited.
\newblock In {\em Trends in representation theory of algebras and related
  topics}, EMS Ser. Congr. Rep., pages 349--397. Eur. Math. Soc., Z\"urich,
  2008.

\bibitem[Iya11]{Iya11}
Osamu Iyama.
\newblock Cluster tilting for higher {A}uslander algebras.
\newblock {\em Adv. Math.}, 226(1):1--61, 2011.

\bibitem[Jas16]{Jas16}
Gustavo Jasso.
\newblock {$n$}-abelian and {$n$}-exact categories.
\newblock {\em Math. Z.}, 283(3-4):703--759, 2016.

\bibitem[JK16]{JK16}
Gustavo Jasso and Sondre Kvamme.
\newblock An introduction to higher {A}uslander-{R}eiten theory.
\newblock {\em arXiv:1610.05458v1}, 2016.

\bibitem[J{\o}r15]{Jor15}
Peter J{\o}rgensen.
\newblock Torsion classes and t-structures in higher homological algebra.
\newblock {\em arXiv:1412.0214v3}, 2015.

\bibitem[Pas17]{Pas17}
Andrea Pasquali.
\newblock Tensor products of higher almost split sequences.
\newblock {\em J. Pure Appl. Algebra}, 221(3):645--665, 2017.

\end{thebibliography}

\end{document}